\documentclass[reqno]{amsart}
\usepackage{amsmath}
\usepackage{amsfonts}
\usepackage{amssymb}
\usepackage{amsthm}
\usepackage{newlfont}
\usepackage{graphicx}

\newtheorem{thm}{Theorem}[section]

\newtheorem{lem}[thm]{Lemma}
\newtheorem{prop}[thm]{Proposition}
\theoremstyle{definition}

\theoremstyle{remark}
\newtheorem{rem}[thm]{Remark}

\numberwithin{equation}{section}
\numberwithin{thm}{section}

\newcommand{\R}{{\mathbb{R}}}

\newcommand{\sgn}{\text{sgn}}
\newcommand{\re}{\text{Re}}
\newcommand{\ed}{\end {document}}

\newcounter{smalllist}

\title[Nonlocal $\alpha$-patch model]{On a one-dimensional $\alpha$-patch model with nonlocal drift and
fractional dissipation}
\author[H. Dong]{Hongjie Dong}
\address[H. Dong]{Division of Applied Mathematics, Brown University,
182 George Street, Providence, RI 02912, USA}
\email{Hongjie\_Dong@brown.edu}
\thanks{H. Dong was partially supported by the NSF under agreements DMS-0800129 and DMS-1056737.}

\author[D. Li]{Dong Li}
\address{Department of Mathematics, University of British Columbia, Vancouver BC Canada V6T 1Z2}%
\email{mpdongli@gmail.com}
\begin{document}

\begin{abstract}
We consider a one-dimensional
nonlocal nonlinear equation of the form: $\partial_t u = (\Lambda^{-\alpha} u)\partial_x u - \nu \Lambda^{\beta}u$
where $\Lambda =(-\partial_{xx})^{\frac 12}$ is the fractional Laplacian and $\nu\ge 0$ is the viscosity
coefficient. We consider primarily the regime $0<\alpha<1$ and $0\le \beta \le 2$ for which the model has
nonlocal drift, fractional dissipation, and captures essential features of the 2D $\alpha$-patch models.
In the critical and subcritical
range $1-\alpha\le \beta \le 2$, we prove global wellposedness for arbitrarily large initial data in
Sobolev spaces. In the full supercritical
range $0 \le \beta<1-\alpha$, we prove formation of singularities in finite time for a class of
smooth initial data. Our proof is based on a novel nonlocal weighted inequality which can be of independent interest.

\end{abstract}
\maketitle
\section{Introduction}

We consider the following nonlocal nonlinear transport equation of
Burgers' type:
\begin{align} \label{e_main}
\partial_t u =(\Lambda^{-\alpha} u) \partial_x u  -\nu \Lambda^{\beta} u,
\end{align}
where $\Lambda =(-\partial_{xx})^{\frac 12}$ is the fractional Laplacian,
$\nu\ge 0$ is the viscosity coefficient, $0<\alpha<1$, and  $0\le \beta\le 2$. In the inviscid case,
this equation can be viewed as a 1D analogy of the 2D $\alpha$-patch problem
$$
\partial_t\theta+u\cdot \nabla\theta=0,\quad u=(-\partial_{x_2},\partial_{x_1})(-\Delta)^{-(1+\alpha)/2}\theta,
$$
which represents an interpolation between the 2D Euler and quasi-geostrophic equations;
see \cite{CFMR05, G08, CCW11} for some discussions.

Equation \eqref{e_main} becomes more singular as $\alpha$ decreases.
It reduces to the inviscid Burger's equation when $\alpha=\nu=0$, in
which case it is well known that solutions may develop gradient
blowup in finite time. When $\alpha=0$ and $\beta\in (0,2]$,
\eqref{e_main} becomes the so-called fractal Burgers' equation,
which is perhaps one of the simplest nonlinear equations with
nonlocal terms. The fractal Burgers' equation was studied in detail
recently; see, for instance, \cite{BFW, ADV, KNS, DDL, MW09, CC10}.
It is known that in the super-critical dissipative case $\beta\in
(0,1)$, with very generic initial data the equation is locally
well-posed and its solution may develops gradient blowup in finite
time. In the critical and sub-critical dissipative case $\beta\in
[1,2]$, the equation is globally well-posed with arbitrary initial
data in suitable Sobolev spaces. Another borderline situation is
when $\alpha=1$ and $\nu=0$, in which case the equation is globally
well-posed; see \cite{OSW08} for a proof in the periodic boundary
condition case. As a matter of fact, in this case in the same spirit
as the Beale--Kato--Majda criterion the solution is regular up to
time $T$ as long as
$$
\int_0^T \|u(t,\cdot)\|_{L^\infty}\,dt<\infty.
$$
This holds for any $T$ since \eqref{e_main} is a transport equation
and $\|u(t,\cdot)\|_{L^\infty}$ is non-increasing.

In this paper, we consider primarily the regime $0<\alpha<1$ and
$0\le \beta \le 2$ for which the model has nonlocal drift,
fractional dissipation, and captures essential features of the 2D
$\alpha$-patch models. We are interested in proving either the
global regularity or finite-time blowup for \eqref{e_main}. In the
general case, the corresponding regularity criterion turns out to be
$$
\int_0^T \|\partial_x\Lambda^{-\alpha}u(t,\cdot)\|_{\infty}\,dt<\infty\quad
\text{or}\quad
\int_0^T \|\Lambda^{1-\alpha}u(t,\cdot)\|_{\infty}\,dt<\infty,
$$
which are not straightforward to verify. See Theorem \ref{thm_M1}
and Remark \ref{rm2.5}.

Let us describe the main results of the paper. In the critical and
sub-critical range $1-\alpha\le \beta \le 2$, we obtain global
wellposedness for arbitrarily large initial data in suitable Sobolev
spaces. In the full supercritical range $0 \le \beta<1-\alpha$, we
prove formation of singularities in finite time for a family of
smooth initial data. These results are in the same spirit of the
results in \cite{KNS, DDL} for the fractal Burgers' equation.

We state our main results of the paper more precisely in the following two theorems.


\begin{thm}[Finite time blowup in the supercritical case] \label{thm2}
For any $0<\alpha<1$, $0\le \beta<1-\alpha$, $2\alpha<\delta<2(1-\beta)$ and $\nu \ge 0$,
 there is a constant $C_{\alpha,\beta,\delta, \nu}>0$
sufficiently large such that the following hold true: \\
For any smooth initial data $u_0 \in C_c^\infty(\mathbb R)$ which is odd in $x$ and satisfies:
\begin{itemize}
\item[a)] $u_0(x) \ge 0$, for any $x\ge 0$;
\item[b)]
\begin{align}
\int_0^{\infty} \frac { u_0(x)} {x^{\delta-\alpha}} \,dx \ge C_{\alpha,\beta,\delta,\nu}
(1+ \|u_0\|_{L_x^1} ), \label{e228_e5}
\end{align}
\end{itemize}
the corresponding solution to \eqref{e_main} blows up in finite time.
\end{thm}

\begin{thm}[Global wellposedness in the critical and supercritical cases] \label{thm3}
Let $0<\alpha<1$, $1-\alpha\le \beta\le 2$, and $\nu > 0$. Assume
$1<p< \frac 1 {\alpha}$ and set $k_0= \max\{ \frac{\frac 1p-\frac 12}{\frac1p-\alpha}, 2\}$. Suppose that
the initial data $u_0 \in L^{p} \cap H^k$ for some integer $k> k_0$.
Then there exists a unique solution $u$ to \eqref{e_main}
in the space $ C([0,\infty), L^p \cap H^k)$ with $u(0)=u_0$.
Moreover, $u\in C((0,T), L^p \cap H^{k^\prime})$ for any $k^\prime \ge k$.
\end{thm}

For the proof of Theorem \ref{thm2},  we consider the evolution of a
weighted integral  of $\Lambda^{-\alpha}u$. The main difficulty is
to show that such quantity satisfies some ordinary differential
inequality, and would blow up in finite time under conditions a) and
b) above. This then implies the blowup of the solution. Before, this
type of approaches can be found, for instance, in \cite{CCF05,
CCF06, dongli2, LR08, LR09, CCG10}. We point out that, in these
papers, weighted integral of solutions themselves are investigated,
which unfortunately does not seem to work in our case. The key new
observation in our case is that instead of working with the solution
$u$ itself, it is more natural to consider certain weighted integral
of the nonlocal quantity $\Lambda^{-\alpha} u$ and establish
positive lower bounds expressed in terms of $\Lambda^{-\alpha} u$.
More precisely we have the following nonlocal weighted inequality,
which could be of independent interest.

\begin{prop}[A nonlocal weighted inequality] \label{prop1}
Let $0<\alpha<1$ and $2\alpha<\delta<2$. There is a constant
$C_{\alpha,\delta}>0$ such that for any odd function $u \in C^1_b(\mathbb R) \cap L^1_x(\mathbb R)$, we have
\begin{align}
\int_0^\infty \frac{\Lambda^{-\alpha} ( \Lambda^{-\alpha} u \cdot \partial_x u ) } {x^{\delta}}
\,dx \ge C_{\alpha,\delta} \int_0^\infty \frac{(\Lambda^{-\alpha} u )^2} {x^{1+\delta}} \,dx. \label{e25_40a}
\end{align}
Here $C_b^1(\mathbb R)=\{ f \in C^1(\mathbb R): \|f\|_{L^\infty_x} + \| f^\prime\|_{L^\infty_x} <\infty\}$.
\end{prop}

\begin{rem}
The parity assumption in Proposition \ref{prop1} can probably be
removed or relaxed by a more elaborate analysis but we will not do
it here. The nontrivial point of Proposition \ref{prop1} is to guess
the correct positive lower bound such as the right-hand side of \eqref{e25_40a}.
We briefly explain the difficulty as follows. By a simple
computation (see e.g. Lemma \ref{lem1}), we have for some constant
$C_{\alpha,\delta}^\prime
>0$,
\begin{align*}
\text{LHS of \eqref{e25_40a}} = C^\prime_{\alpha,\delta}
\int_0^{\infty} \frac{ \Lambda^{-\alpha} u \cdot \partial_x u }{
x^{\delta-\alpha} }dx.
\end{align*}
From this and a scaling heuristic, one is led to conjecture the
inequality
\begin{align}
\int_0^{\infty} \frac{ \Lambda^{-\alpha} u \cdot \partial_x u }{
x^{\delta-\alpha} }dx \gtrsim \int_0^{\infty} \frac {u^2}
{x^{1+\delta-2\alpha}} dx. \label{conj_ineq}
\end{align}

The inequality \eqref{conj_ineq} is close in spirit to the type of
inequalities used in \cite{CCF05, CCF06, dongli2, LR08, LR09,
CCG10}. However a preliminary calculation shows that the inequality
\eqref{conj_ineq} is probably false unless some negative terms are
added on the right-hand side of \eqref{conj_ineq}. To circumvent this difficulty
we prove \eqref{e25_40a}. Note that these two lower bounds are
\emph{not} equivalent although they obey the same scaling relations.
\end{rem}

For the proof of Theorem \ref{thm3},  when $\beta<2$ we adapt an
idea of the  non-local maximum principle for a suitably chosen
modulus of continuity. This method was first used by Kiselev,
Nazarov, and Volberg in \cite{kiselev}, where they established the
global regularity for the 2D critical dissipative quasi-geostrophic
equations with periodic $C^\infty$ data. We also refer the reader to
\cite{Ki11} and the references therein for further applications and
development of this method. In the case $\beta=2$, this method does
not seem to be applicable, and we use a different approach.

The remaining part of the paper is organized as follows. In the next
section, we prove the local wellposedness, continuation criteria for
\eqref{e_main}, and Theorem \ref{thm3}. In section \ref{sec3} we
present several auxiliary lemmas, which will be used in the proofs of
Proposition \ref{prop1} and Theorem \ref{thm2} in the last section.

We close this introduction by setting up some

\subsection*{Notations} For any two quantities $A$ and $B$,  we use $A \lesssim B$ (resp. $A
\gtrsim B$ ) to denote the inequality $A \leq CB$ (resp. $A \geq
CB$) for a generic positive constant $C$. The dependence of $C$ on
other parameters or constants are usually clear from the context and
we will often suppress  this dependence.  The value of $C$ may
change from line to line.  For any function $f:\; \mathbb R\to
\mathbb R$, we use $\|f\|_{L^p}$ or sometimes $\|f\|_p$ to denote
the  usual Lebesgue $L^p$ norm of a function for $1 \le p \le
\infty$.

The fractional Laplacian operator $\Lambda^s$, $s\in \mathbb R$ is
defined via Fourier transform as
\begin{align*}
\mathcal F (  |\nabla|^s f ) (\xi) = |\xi|^s (\mathcal F f)(\xi),
\qquad \xi \in \mathbb R.
\end{align*}
The homogeneous Sobolev norm $\dot H^{s}$ for any $s \ge 0$ is
defined as $\| f \|_{\dot H^{s} } = \| |\nabla|^s f \|_{2}$ or more
explicitly:
\begin{align*}
\| f \|_{\dot H^s} = \Bigl( \int_{\mathbb R} |\xi|^{2s} |(\mathcal F
f)(\xi)|^2 d\xi \Bigr)^{\frac 12}.
\end{align*}

We will need to use the Littlewood--Paley (LP) frequency projection
operators. For simplicity we shall fix the notations on $\mathbb R$,
but it is straightforward to define everything in $\mathbb R^d$ for
any $d\ge 1$. To fix the notation, let $\phi \in
C_0^\infty(\mathbb{R})$ and satisfy
\begin{equation}\nonumber
0 \leq \phi \leq 1,\quad \phi(x) = 1\ {\text{ for}}\ |x| \leq
1,\quad \phi(x) = 0\ {\text{ for}}\ |x| \geq 2.
\end{equation}
For two real positive numbers $\alpha < \beta$, define the frequency
localized (LP) projection operator $P_{\alpha <\cdot<\beta}$ by
\begin{equation}\nonumber
P_{\alpha <\cdot<\beta}f = \mathcal{F}^{-1}\big([\phi(\beta^{-1}\xi)
- \phi(\alpha^{-1}\xi)]\mathcal{F}(f)\big).
\end{equation}
Here $\mathcal{F}$ and $\mathcal{F}^{-1}$ denote the Fourier
transform and its inverse transform, respectively. Similarly, the
operators $P_{< \alpha}$ and $P_{> \beta}$ are defined by
\begin{equation}\nonumber
P_{<\beta}f =
\mathcal{F}^{-1}\big(\phi(\beta^{-1}\xi)\mathcal{F}(f)\big),
\end{equation}
and
\begin{equation}\nonumber
P_{> \alpha}f = \mathcal{F}^{-1}\big([1 -
\phi(\alpha^{-1}\xi)]\mathcal{F}(f)\big).
\end{equation}

We recall the following Bernstein estimates:  for any $1\le p\le
q\le \infty$ and dyadic $N>0$,
\begin{align*}
\| P_{<N} f\|_{L^q (\mathbb R)} \lesssim N^{\frac 1p -\frac 1q} \|
f\|_{L^p(\mathbb R)}.
\end{align*}

\section{Local and global regularity}
We first state the following positivity lemma which is a simple variant
of Lemma 2.5 in \cite{CC04}. We include the proof here for the sake of completeness.
\begin{lem}[Positivity lemma]
Let $F: \mathbb R\to \mathbb R$ be a nondecreasing function. Assume $0\le \beta \le 2$. Then
for any $\theta:\mathbb R\to \mathbb R$ such that $\Lambda^\beta \theta$ is well-defined and
$F(\theta) \Lambda^{\beta} \theta \in L^1(\mathbb R)$, we have
\begin{align*}
\int_{-\infty}^{\infty} F(\theta) \Lambda^\beta \theta \,dx \ge 0.
\end{align*}
In particular for any $1\le p<\infty$,
\begin{align}
\int_{-\infty}^{\infty} |\theta|^{p-2} \theta \Lambda^{\beta} \theta \,dx \ge 0, \label{Ma4_e1}
\end{align}
provided that the integral is well-defined.
\end{lem}

\begin{proof}
Without loss of generality, we assume $F$ is a smooth function. In
the general case one can mollify $F$ and deduce the result by a
limiting argument. Consider first $\beta=2$, in this case we just
integrate by parts and obtain
\begin{equation*}
\int_{-\infty}^\infty F(\theta) (-\partial_{xx}) \theta \,dx \notag \\
= \int_{-\infty}^\infty F^{\prime}(\theta) (\partial_x \theta)^2 \,dx \ge 0,
\end{equation*}
since $F$ is non-decreasing.

Now we assume $0<\beta<2$. Recall that for $0<\beta<1$, we have
\begin{align}
(\Lambda^\beta g)(x) = C_\beta \lim_{\epsilon\to 0} \int_{|y-x|>\epsilon}
\frac{g(x)-g(y)}{|x-y|^{1+\beta}} \,dy \label{tp100_1}
\end{align}
for some constant $C_\beta>0$.

For $\beta=1$ one can use $\Lambda g= H \partial_x g$ ($H$ is the Hilbert transform) and
integration parts to
show that the formula \eqref{tp100_1} still holds. For smooth $g$, an equivalent formula (without $\epsilon$-limit)
is given by the expression
\begin{align*}
(\Lambda g)(x) = C \int_{|y-x|<1} \frac{g(x)-g(y) {+} g^{\prime}(x)
(y-x)} {|y-x|^2} \,dy + C\int_{|y-x|>1} \frac{g(x)-g(y)} {|y-x|^2}
\,dy.
\end{align*}

Similarly for $1<\beta<2$ one
can use $\Lambda^{\beta} g = -\Lambda^{\beta-2} \partial_{xx} g$, fractional representation of
the Riesz potential $\Lambda^{-(2-\beta)}$ and integration by
parts (twice) to show that \eqref{tp100_1} also holds in this case. In this case a formula equivalent
to \eqref{tp100_1} is given by
\begin{align*}
(\Lambda^{\beta} g)(x) = C_{\beta} \int_{\mathbb R}
\frac{g(x)-g(y){+}g^{\prime}(x) (y-x)} { |x-y|^{1+\beta}} \,dy.
\end{align*}

In all cases we shall just use fractional representation of
$\Lambda^\beta$ as in \eqref{tp100_1}. Clearly
\begin{align*}
 &\int_{-\infty}^\infty F(\theta) \Lambda^{\beta} \theta \,dx \\
=& C_{\beta} \int_{\mathbb R^1_x} F(\theta(x))
\lim_{\epsilon \to 0} \int_{|y-x|>\epsilon} \frac{\theta(x)-\theta(y)}{|x-y|^{1+\beta}} \,dy \,dx \\
= & \lim_{\epsilon \to 0} C_{\beta} \int_{\mathbb R^1_x} F(\theta(x))
\int_{|y-x|>\epsilon} \frac{\theta(x)-\theta(y)}{|x-y|^{1+\beta}} \,dy \,dx.
\end{align*}
Symmetrizing the above integral in $x$ and $y$, we obtain
\begin{align*}
&\int_{-\infty}^\infty F(\theta) \Lambda^{\beta} \theta \,dx \\
=& \frac 12 C_{\beta}
\lim_{\epsilon \to 0}
\int_{\mathbb R^1_x}\int_{|y-x|>\epsilon} ( F(\theta(x))-F(\theta(y)) )
 \frac{\theta(x)-\theta(y)}{|x-y|^{1+\beta}} \,dx \,dy,
\end{align*}
which is clearly non-negative since $F$ is a non-decreasing function.
\end{proof}

Next we state and prove the local wellposedness and a continuation
criterion for Equation \eqref{e_main}. The main issue is the
nonlocal drift term $\Lambda^{-\alpha} u$ which induces some
integrability constraints on $u$, see Remark \ref{rm23} and Remark
\ref{rm24} below.
\begin{thm}[Local wellposedness and continuation criterion] \label{thm_M1}
Suppose $0<\alpha<1$, $\nu\ge0$ and $0\le \beta \le 2$ in \eqref{e_main}. Assume
$1<p< \frac 1 {\alpha}$ and set $k_0= \max\{ \frac{\frac 1p-\frac 12}{\frac1p-\alpha}, 2\}$.
the initial data $u_0 \in L^{p} \cap H^k$ for some integer $k> k_0$.
Then there exists $T>0$ and a unique solution $u$ to \eqref{e_main}
in the space $ C([0,T), L^p \cap H^k)$ with $u(0)=u_0$.  The solution can be continued beyond
any $T^\prime\ge T$ if
\begin{align}
\int_0^{T^\prime}
 \| \partial_x \Lambda^{-\alpha} u (t) \|_{\infty}
  \,dt <+\infty. \label{Mar6_e9}
\end{align}
Moreover if $\nu>0$ and $0<\beta\le 2$, then the solution has additional regularity,
i.e. $u\in C((0,T), L^p \cap H^{k^\prime})$ for any $k^\prime \ge k$.
\end{thm}

\begin{rem} \label{rm23}
In Theorem \ref{thm_M1}, to have local wellposedness, we require the
initial data $u_0$ to lie in  $L^p$ for some $1<p{\le} \frac
1{\alpha}$. For $0<\alpha\le \frac 12$, one can just work with pure
$H^k$ spaces. However for $\frac 1 2 < \alpha<1$,
it is essential that $u_0$ belongs $L^p$ for some small $p$ in order
to control the low frequency part. The regularity condition $k\ge
k_0$ comes from bounding the quantity
 $\| \partial_x u\|_{\frac 1 {\alpha}}$ when we perform contraction estimates in $C_t^0 L^p_x$
 (see e.g. \eqref{Mar6_e5}).
\end{rem}

\begin{rem} \label{rm24}
Instead of the space $L^p\cap H^k$, one can also choose the space
$W^{k,p}$ for $1<p<\frac 1{\alpha}$ when $\frac 12 <\alpha <1$.
Another possibility is the space $H^{-\delta} \cap H^{k}$ for some
$\delta>0$. {Of course, the choice of $k$ in Theorem \ref{thm_M1} is
not optimal.} We shall not dwell on these issues here. In any case
these spaces provide natural $L^\infty$ bounds on the drift term
$\Lambda^{-\alpha} u$ and produce classical solutions.
\end{rem}

\begin{proof}[Proof of Theorem \ref{thm_M1}]
The proof is more or less a standard application of energy estimates
and hence we only sketch the details here. Define $u^{(0)}=u_0$, and
for $n\ge 0$ inductively define the iterates $u^{(n+1)}$ such that
\begin{align} \label{Ma3_e1}
\begin{cases}
\partial_t u^{(n+1)} = (\Lambda^{-\alpha} u^{(n)}) \partial_x u^{(n+1)} - \nu \Lambda^\beta u^{(n+1)},
\\
u^{(n+1)}(0)=u_0.
\end{cases}
\end{align}
Note that the system \eqref{Ma3_e1} is linear in $u^{(n+1)}$. By a simple induction we
have $u^{(n)} \in C([0,\infty), L^p \cap H^k)$ for all $n\ge 0$ and furthermore for some
$T>0$ depending only on $u_0$,
\begin{align}
\sup_{n\ge 0} \Bigl( \|u^{(n)}\|_{L_t^\infty H^k([0,T))}+
\| u^{(n)} \|_{L_t^\infty L^p([0,T))} \Bigr) \le A_1<\infty, \label{Ma3_e3}
\end{align}
where $A_1>0$ is a constant. Note that in deriving \eqref{Ma3_e3}, we have used
\eqref{Ma4_e1}.

The next step is to show contraction in $C([0,T_1], L_x^p)$ for some $T_1$ possibly smaller than $T$.
Define $w^{(n+1)} = u^{(n+1)} - u^{(n)}$. By \eqref{Ma3_e1}, we have
\begin{align*}
\partial_t w^{(n+2)} = \Lambda^{-\alpha} u^{(n+1)} \partial_x w^{(n+2)} + \Lambda^{-\alpha}
w^{(n+1)} \partial_x u^{(n+1)} - \nu \Lambda^{\beta} w^{(n+2)}.
\end{align*}

Multiplying both sides by $|w^{(n+2)}|^{p-2} w^{(n+2)}$, integrating by parts and using
again \eqref{Ma4_e1}, we have
\begin{align*}
\frac 1 p  \frac d {dt}
\int_{\mathbb R} |w^{(n+2)}|^p \,dx
&\le \frac 1p \int_{\mathbb R} |\partial_x \Lambda^{-\alpha} u^{(n+1)} | |w^{(n+2)}|^p \,dx \notag \\
&\qquad + \int_{\mathbb R} |\Lambda^{-\alpha} w^{(n+1)} | |w^{(n+2)}|^{p-1}
|\partial_x u^{(n+1)} | \,dx.
\end{align*}
By H\"older and Sobolev embedding,
\begin{align}
\frac d {dt} \Bigl( \| w^{(n+2)} (t) \|^p_p \Bigr)
& \lesssim \| \partial_x \Lambda^{-\alpha} u^{(n+1)} (t) \|_\infty \| w^{(n+2)} (t) \|_p^p \notag \\
& \qquad + \| \Lambda^{-\alpha} w^{(n+1)} (t) \|_{{(\frac 1p -\alpha)}^{-1}}
\| w^{(n+2)}(t) \|_p^{p-1} \| \partial_x u^{(n+1)} (t) \|_{\frac 1 {\alpha}} \notag\\
& \lesssim \| u^{(n+1)}(t) \|_{H^2} \| w^{(n+2)}(t) \|_p^p \notag \\
& \qquad + \| w^{(n+1)}(t) \|_p \cdot \|w^{(n+2)}(t)\|_p^{p-1} \cdot
\| \partial_x u^{(n+1)} (t) \|_{\frac 1 {\alpha}}. \label{Mar6_e1}
\end{align}

To bound $\| \partial_x u^{(n+1)}(t) \|_{\frac 1 {\alpha}}$ we discuss two cases. If $0<\alpha\le \frac 12$,
then one can use Sobolev embedding to get
\begin{align}
\| \partial_x u^{(n+1)} (t) \|_{\frac 1 {\alpha}} \lesssim \| u^{(n+1)}(t) \|_{H^2}. \label{Mar6_e3}
\end{align}

If $\frac 12<\alpha<1$, then we use the interpolation inequality
\begin{align*}
\|g\|_{\frac 1 {\alpha}} \lesssim \| g\|_p^{1-\theta}
\| g\|_2^{\theta},\quad \theta=\frac{\frac 1p -\alpha}{\frac 1p -\frac12}
\end{align*}
to get
\begin{align}
\| \partial_x u^{(n+1)} \|_{\frac 1{\alpha}} &\lesssim \| u^{(n+1)}\|_p + \sum_{N>1} N
\| P_N u^{(n+1)}\|_p^{1-\theta} \| P_N u^{(n+1)} \|_2^{\theta} \notag\\
&\lesssim \| u^{(n+1)} \|_p + \| u^{(n+1)} \|_{H^k}, \label{Mar6_e5}
\end{align}
where we have used the fact that $k>k_0 \ge \frac 1 {\theta}$. Plugging the bounds
\eqref{Mar6_e3}--\eqref{Mar6_e5} into \eqref{Mar6_e1} and using \eqref{Ma3_e3}, we obtain
\begin{align*}
\frac {d}{dt} \Bigl( \| w^{(n+2)}(t) \|_p^p \Bigr) \lesssim A_1\cdot \Bigl(\| w^{(n+2)}(t) \|_p^p
+ \|w^{(n+1)}(t)\|_p \cdot \| w^{(n+2)}(t) \|_p^{p-1} \Bigr).
\end{align*}
By using the fact $w^{(n+1)}(0)=0$ for all $n\ge 0$, and choosing $T_1$ smaller than $T$ if necessary,
we get for some $0<\tilde {\theta}<1$ that
\begin{align*}
\|w^{(n+2)} (t) \|_{L_t^\infty L_x^p([0,T_1])} \le \tilde {\theta}
\|w^{(n+1)} (t) \|_{L_t^\infty L_x^p([0,T_1])},\quad \forall\, n\ge 0.
\end{align*}
Hence the sequence $u^{(n)}$ has a strong limit in $C([0,T_1], L_x^p)$.

By interpolation one also obtains strong convergence and the limit
solution $u$ in $C([0,T_1], H^{k-1})$. By a standard argument (one
has to discuss separately the case $\nu=0$ and the case $\nu>0$, and
the fact that strong continuity in $H^k$ is equivalent to weak
continuity together with norm continuity), one can show that $u \in
C([0,T_1], H^k)$. We omit the details.

Now we turn to the proof of the continuation criterion \eqref{Mar6_e9}. By \eqref{Ma4_e1}, it is easy to show that
\begin{align}
\frac {d} {dt} ( \| u(t)\|_p^p) \lesssim \| \partial_x \Lambda^{-\alpha} u(t) \|_{\infty} \| u(t) \|_p^p,
\label{Mar6_e13}
\end{align}
and hence $\|u(t)\|_p$ is controlled by \eqref{Mar6_e9}. It remains
to control the $H^k$ norm. From \eqref{e_main}, using integration by
parts and H\"older, we have

\begin{align}
&\frac d {dt} \Bigl( \| \partial_x^k u(t) \|_2^2 \Bigr)\notag\\
& \lesssim \int_{\mathbb R} \partial_x^k( \Lambda^{-\alpha} u \partial_x u ) \partial_x^k u \,dx \notag \\
& \lesssim \| \partial_x \Lambda^{-\alpha} u(t) \|_\infty \| \partial_x^k u(t) \|_2^2 \notag \\
& \quad + \| \partial_x^k u(t) \|_2\sum_{1\le l\le k} \|
\partial_x^{l-1} \partial_x \Lambda^{-\alpha} u(t)
\|_{\frac{2(k-1+\alpha)}{l-1}} \cdot \|
\partial_x^{k-l}\Lambda^{\alpha} \partial_x \Lambda^{-\alpha} u
\|_{\frac{2(k-1+\alpha)}{k-l+\alpha}}. \label{Ma6_e11}
\end{align}

By using the Gagliardo--Nirenberg inequalities, we have for any $1\le l \le k$,
\begin{align*}
&\| \partial_x^{l-1} \partial_x \Lambda^{-\alpha} u(t) \|_{\frac{2(k-1+\alpha)}{l-1}}
\lesssim \| \partial_x \Lambda^{-\alpha} u(t) \|_{\infty}^{1-\frac{l-1}{k-1+\alpha}}
\cdot \| \partial_x^k u(t)  \|_2^{\frac{l-1}{k-1+\alpha}}, \\
& \| \partial_x^{k-l} \Lambda^{\alpha} \partial_x \Lambda^{-\alpha} u \|_{\frac{2(k-1+\alpha)}{k-l+\alpha}}
\lesssim \| \partial_x \Lambda^{-\alpha} u \|_{\infty}^{1-\frac{k-l+\alpha}{k-1+\alpha}}
\cdot \| \partial_x^{k}   u \|_2^{\frac{k-l+\alpha}{k-1+\alpha}}.
\end{align*}

Plugging the above estimates into \eqref{Ma6_e11}, and we obtain
\begin{align*}
\frac{d}{dt} \Bigl( \| \partial_x^k u(t) \|_2^2 \Bigr) & \lesssim
\| \partial_x \Lambda^{-\alpha} u(t) \|_{\infty} \cdot
\|\partial_x^k u(t)\|_2^2.
\end{align*}
Together with \eqref{Mar6_e13} this easily yields \eqref{Mar6_e9}.

Finally in the case $\nu>0$ and $0<\beta\le 2$, one can use the
theory of mild solutions or simple energy estimates to gain
additional regularity. See, for instance, \cite{D08}. We omit the
details.
\end{proof}
\begin{rem}
                    \label{rm2.5}
From the logarithmic type bound
$$
\|\partial_x\Lambda^{-1}f\|_{L^\infty}\leq C\Big[1+\|f\|_{L^\infty}\log(e+\|f
\|_{H^1})+\|f\|_{L^2}\Big],
$$
it is easily seen that the condition \eqref{Mar6_e9} can be replaced by
$$
\int_0^{T'}\|\Lambda^{1-\alpha} u(t)\|_{\infty}\,dt<\infty.
$$
However, we shall not use this fact in the sequel.
\end{rem}

For the proof of Theorem \ref{thm3}, when $\beta<2$ we shall use the
idea of the non-local maximum principle for a suitably chosen
modulus of continuity. This method was first used by Kiselev,
Nazarov, and Volberg in \cite{kiselev}, where they proved the global
regularity for the 2D critical dissipative quasi-geostrophic
equations with periodic $C^\infty$ data. In the borderline case when
$\beta=2$, this argument does not seem to work  since \eqref{eq1.23}
below no longer holds. Therefore, we will use a different approach
in this case.

We say a function $f$ has modulus of continuity $\omega$ if
$|f(x)-f(y)|\leq \omega(|x-y|)$, where $\omega$ is an increasing
continuous function $\omega$ : $[0,+\infty)\to [0,+\infty)$ and $\omega(0)=0$.
We say $f$ has strict modulus of continuity $\omega$ if the
inequality is strict for $x\neq y$.

In what follows, we will choose a concave function $\omega$ satisfying
\begin{equation}
                    \label{eq10.18}
\omega'>0,\quad \omega'(0)<+\infty,\quad
\lim_{\xi\to 0^+} \omega''(\xi)=-\infty,
\end{equation}

Owing to Theorem \ref{thm_M1} and the Sobolev imbedding theorem, we may assume $\theta_0\in H^{20}\cup C^\infty$. Because of
the scaling property of \eqref{e_main}, for any $\lambda>0$,
$$
u_\lambda(t,x)=\lambda^{\alpha+\beta-1}u(\lambda^{\beta}t,\lambda x)
$$
is also a solution of \eqref{e_main} with initial data
$\lambda^{\alpha+\beta-1}u_0(\lambda^\beta x)$. Thus if we can show that
$u_\lambda$ is a global solution, the same remains true for $u$.
Note that for any $\omega$ satisfying \eqref{eq10.18} we can always
find a constant $\lambda>0$ such that $\omega(\xi)$ is a {\em strict}
modulus of continuity of $\lambda^{\alpha+\beta-1}u_0(\lambda^\beta x)$ provided that $\alpha+\beta> 1$. While in the critical case, i.e., $\alpha+\beta=1$, this still holds for any unbounded $\omega$ satisfying \eqref{eq10.18}.

We shall show that for suitably chosen $\omega$, the modulus of
continuity is preserved for all the time. This gives $\|\partial_x u(t)\|_{\infty}\le \omega'(0)$, which together with the uniform boundedness of $\|u(t)\|_{\infty}$ and Theorem \ref{thm_M1} implies Theorem \ref{thm3}.  Following the argument in \cite{kiselev} and \cite{dongdu}, the strict modulus of continuity is preserved at
least for a short time. Also it is clear that if $u(t,\cdot)$ has strict modulus of continuity $\omega$ for all $t\in [0,T)$, then $\theta$ is smooth up to $T$
and $u(T,\cdot)$ has modulus of continuity $\omega$ by continuity. Therefore, to show that the modulus of continuity is
preserved for all the time, it suffices to rule out the case that
\begin{equation*}
\sup_{x\neq y}\frac {u(T,x)-u(T,y)}{\omega(|x-y|)}=1,
\end{equation*}
which in turn implies that there exist two different points $x,y\in \R$ satisfying
\begin{equation*}
\theta(T,x)-\theta(T,y)=\omega(|x-y|)
\end{equation*}
due to  \eqref{eq10.18} and the decay of $u$ at the spatial infinity (see \cite{dongdu}).
This possibility can be eventually ruled out if we are able to chose
suitable $\omega$ such that under the conditions above we have
\begin{equation}
                            \label{eq24.5.14}
\frac {\partial}{\partial t}(u(T,x)-u(T,y))<0.
\end{equation}
To this end, we need the following two lemmas.

\begin{lem}             \label{lem3.4}
Assume that $u$ has modulus of continuity $\omega$, which is a concave function. Then for any $\alpha\in (0,1)$, $\Lambda^{-\alpha}u$ has modulus of continuity
\begin{equation}
                                                \label{eq31.2.45}
\Omega(r):=C_{\alpha}\Big(r^\alpha \omega(r)+r\int_r^\infty \frac {w(s)} {s^{2-\alpha}}\,ds\Big).
\end{equation}
\end{lem}
\begin{proof}
For any $r>0$, take $x,y\in \R$ such that $|x-y|=r$. Without loss of generality, we may assume $x=-r/2$ and $y=r/2$. Then
$$
|\Lambda^{-\alpha} u(x)-\Lambda^{-\alpha} u(y)|=C_\alpha\Big|\int_{-\infty}^\infty \frac {u(s)-u(0)} {|r/2+s|^{1-\alpha}}\,ds-\int_{-\infty}^\infty \frac {u(s)-u(0)} {|-r/2+s|^{1-\alpha}}\,ds\Big|
$$
$$
\le I_1+I_2.
$$
Here
$$
I_1:=C_\alpha\Big|\int_{-r}^r \frac {u(s)-u(0)} {|r/2+s|^{1-\alpha}}\,ds\Big|
+C_\alpha\Big|\int_{-r}^r \frac {u(s)-u(0)} {|-r/2+s|^{1-\alpha}}\,ds\Big|\le C_\alpha r^\alpha \omega(r),
$$
and
$$
I_2:=C_\alpha \left(\int_{-\infty}^{-r}+\int_{r}^\infty\right)|u(s)-u(0)|\Big|
\frac 1 {|r/2+s|^{1-\alpha}}-\frac 1 {|-r/2+s|^{1-\alpha}}\Big|
\le C_{\alpha}r\int_r^\infty \frac {w(s)} {s^{2-\alpha}}\,ds,
$$
where in the last inequality we have used the mean value theorem. The lemma is proved.
\end{proof}

\begin{lem}
                            \label{lem3.5}
Assume $u$ and $\Lambda^{-\alpha}u$ have moduli of continuity $\omega$ and
$\Omega$ respectively, and $u(x)-u(y)=\omega(|x-y|)$ for
some $x\neq y$.

i) Then we have
\begin{equation*}
|\Lambda^{-\alpha}u(x) \partial_x u(x)-\Lambda^{-\alpha}u(y) \partial_x u(y)|\leq
\Omega(|x-y|)\omega'(|x-y|).
\end{equation*}

ii) For any $\beta\in (0,2)$, it holds that
\begin{align}
-\Lambda^\beta u(x)+\Lambda^\beta u(y)\leq&
C_\beta\int_0^{\frac {r} 2}\frac
{\omega(r+2s)+\omega(r-2s)-2\omega(r)} {r^{1+\beta}}\,ds\nonumber\\
&+C_\beta\int_{\frac {r} 2}^\infty\frac
{\omega(2s+r)-\omega(2s-r)-2\omega(r)} {r^{1+\beta}}\,ds,   \label{eq1.23}
\end{align} where $r=|x-y|$.
\end{lem}
\begin{proof}
See \cite{kiselev}, or \cite{D08}.
\end{proof}

Due to Lemmas \ref{lem3.4} and \ref{lem3.5}, the left-hand side  of
\eqref{eq24.5.14} is less than or equal to $I_3+I_4+I_5$, where
\begin{align*}
I_3=\Omega(r)\omega'(r),\quad
I_4=C_\beta\int_0^{\frac {r} 2}\frac
{\omega(r+2s)+\omega(r-2s)-2\omega(r)} {r^{1+\beta}}\,ds,\\
I_5=C_\beta\int_{\frac {r} 2}^\infty\frac
{\omega(2s+r)-\omega(2s-r)-2\omega(r)} {r^{1+\beta}}\,ds,\quad
r=|x-y|,
\end{align*}
where $\Omega(r)$ is given in \eqref{eq31.2.45}.
With concave $\omega$, both $I_4$ and $I_5$ are strictly
negative.

Now we are ready to prove Theorem \ref{thm3}.

\begin{proof}[Proof of Theorem \ref{thm3}]
Let $\delta>0$ be a small number to be specified later. We consider three cases separately.

{\em Case 1: $\beta=1-\alpha$.} Set for $r>0$,
\begin{equation*}
\omega''(r)=-\frac{\delta}{r^{\alpha}+r^2},\quad
\omega'(r)=-\int_r^\infty \omega''(s)\,ds,\quad \omega(0)=0,
\end{equation*}
which clearly satisfies \eqref{eq10.18}. Since $\alpha\in (0,1)$, it is easily seen that, for
$r\ll 1$,
\begin{align}
                                                \label{eq24.9.26}
\omega'(r)\sim \delta,\quad \omega(r)\sim \delta
r,\quad\Omega(r)\sim \delta r,\nonumber\\
I_3=\Omega\omega'\sim \delta^2 r,\quad I_4\leq C\delta
r^{2-\beta}\omega''(r)\sim -\delta r^{2-\beta-\alpha}.
\end{align}
The last inequality is because of the mean value theorem and the
monotonicity of $\omega''$. Similarly, for $r\gg 1$,
\begin{align}
                                                \label{eq24.9.27}
\omega'(r)&\sim \delta r^{-1},\quad \omega(r)\sim
\delta\log r,\quad\Omega(r)\sim \delta r^\alpha\log r,\nonumber\\
I_3&=\Omega\omega'\sim\delta^2 r^{\alpha-1}\log r,\quad I_5\leq
-C\delta r^{-\beta}\log r.
\end{align}
For the bound of $I_5$, we have used $\omega(2s+r)-\omega(2s-r)\le \omega(2r)$.
Recall that $\alpha+\beta= 1$. Here and in what follows, the constant $C$ is independent of $\delta$.
From \eqref{eq24.9.26}, \eqref{eq24.9.27} and the continuity, we
can choose $\delta$ sufficiently small such that $I_3+I_4+I_5$ is
strictly negative on $(0,+\infty)$. Thus, we obtain
\eqref{eq24.5.14}.

{\em Case 2: $\beta\in (1-\alpha,2)$.}
Set for $r>0$,
\begin{equation*}
\omega''(r)=-\frac{\delta}{r^{\alpha}+r^5},\quad
\omega'(r)=-\int_r^\infty \omega''(s)\,ds,\quad \omega(0)=0,
\end{equation*}
which clearly satisfies \eqref{eq10.18}. We still have \eqref{eq24.9.26} for
$r\ll 1$. For $r\gg 1$,
\begin{align}
                                                \label{eq24.9.27b}
\omega'(r)&\sim \delta r^{-4},\quad \omega(r)\sim
\delta,\quad\Omega(r)\sim \delta r^\alpha,\nonumber\\
I_3&=\Omega\omega'\sim\delta^2 r^{\alpha-4},\quad I_5\leq
-C\delta r^{-\beta}.
\end{align}
Recall that $\alpha+\beta\in (1,4)$.
From \eqref{eq24.9.26}, \eqref{eq24.9.27b} and the continuity, we
can choose $\delta$ sufficiently small such that $I_3+I_4+I_5$ is
strictly negative on $(0,+\infty)$. Therefore, we still obtain
\eqref{eq24.5.14} in this case.

{\em Case 3: $\beta=2$.} In this case, we shall use a different
argument since the bound \eqref{eq1.23} does not hold any more. Due
to the classical Sobolev theory for second-order parabolic
equations\footnote{See also Remark \ref{rmdifferent} for a different
proof.}, it suffices for us to show that
\begin{equation}
                                \label{eq1.29}
\|\Lambda^{-\alpha}u\|_{L_t^\infty L_x^\infty([0,T])}<\infty
\end{equation}
for any finite $T>0$. Indeed, by the classical theory \eqref{eq1.29} guarantees that
$u,\partial_x,\partial_{xx}u, u_t\in L^2([0,T]\times \mathbb R)$, which together with the Sobolev imbedding theorem implies \eqref{Mar6_e9}.

In order to prove \eqref{eq1.29}, first we note that by the maximum principle, for any $t>0$,
\begin{equation}
                                        \label{eq1.52}
\| u(t) \|_\infty\le \| u(0) \|_\infty.
\end{equation}
Using the Littlewood--Paley projectors, the Bernstein inequality, and \eqref{eq1.52}, we have
\begin{align}
                            \label{eq2.36}
\| \Lambda^{-\alpha} u(t)\|_\infty &\lesssim \| u(t) \|_\infty + \sum_{\text{N dyadic}: N\le 1} \| \Lambda^{-\alpha} P_{N} u(t) \|_{\infty}\nonumber \\
& \lesssim \|u(0)\|_\infty+ \sum_{N\le 1} N^{-\alpha} N^{\frac 1 p} \| u(t)\|_{p}\nonumber\\
& \lesssim \|u(0)\|_\infty+\|u(t)\|_{p},
\end{align}
where $p$ is from Theorem \ref{thm_M1}, and in third inequality we used
the fact that $\alpha<\frac 1 p$. Multiplying both sides of \eqref{e_main} by $|u|^{p-2}u$ and integrating by parts, we have
\begin{align}
                                \label{eq2.32}
\frac 1 p  \frac d {dt}
\int_{\mathbb R} |u|^p \,dx
&= \int_{\mathbb R} \Lambda^{-\alpha} u  \partial_x u |u|^{p-2} u \,dx+ \nu\int_{\mathbb R} \partial_{xx}u |u|^{p-2}u \,dx\nonumber\\
&= I_1+I_2- \frac{4\nu (p-1) }{p^2}\int_{\mathbb R} \Big(\partial_x |u|^{\frac p 2}\Big)^2 \,dx,
\end{align}
where
$$
I_1:=\int_{\mathbb R} \big(P_{<1}\Lambda^{-\alpha} u\big)  \partial_x u |u|^{p-2} u \,dx,\quad
I_2:=\int_{\mathbb R} \big(P_{\ge 1}\Lambda^{-\alpha} u\big)  \partial_x u |u|^{p-2} u \,dx.
$$
To estimate $I_1$, we integrate by parts and use H\"older's inequality, the Bernstein inequality, and \eqref{eq1.52} to get
\begin{align}
                                \label{eq2.06}
|I_1|&\le \int_{\mathbb R} |P_{<1}\partial_x \Lambda^{-\alpha} u|  |u|^{p} \,dx\nonumber\\
&\le \left\|P_{<1}\partial_x \Lambda^{-\alpha} u(t)\right\|_p \left\|u(t)\right\|_p^{p-1} \left\|u(t)\right\|_\infty\le C\left\|u(t)\right\|_p^{p}.
\end{align}
To bound $I_2$, by H\"older's inequality, the Bernstein inequality, and \eqref{eq1.52} we have
\begin{align}
                                \label{eq2.26}
|I_2|&\le C\left\|P_{\ge 1} \Lambda^{-\alpha} u(t)\right\|_\infty \left\|\partial_x |u|^{\frac p 2}\right\|_2 \left\||u|^{\frac p 2}\right\|_2\nonumber\\
&\le C\left\|u(t)\right\|_\infty \left\|\partial_x |u|^{\frac p 2}\right\|_2 \left\||u|^{\frac p 2}\right\|_2\nonumber\\
&\le C\left\|u(0)\right\|_\infty \left\|\partial_x |u|^{\frac p 2}\right\|_2 \left\||u|^{\frac p 2}\right\|_2.
\end{align}
Combining \eqref{eq2.32}, \eqref{eq2.06}, \eqref{eq2.26}, and using Young's inequality and Gronwall's inequality, we easily get
$\|u(t)\|_p\le Ce^{Ct} $, which together with \eqref{eq2.36} yields \eqref{eq1.29}.

The theorem is proved.
\end{proof}
\begin{rem} \label{rmdifferent}
A slightly different proof for $\beta=2$ is possible and we sketch
it below for the sake of completeness. Multiplying both sides of
\eqref{e_main} by $-\partial_{xx}u$ and integrating by parts, we
have
\begin{align*}
&\frac 1 2\frac d {dt} \int (\partial_x u)^2\, dx \\& = - \int (P_{<1}
\Lambda^{-\alpha} u) \frac 12 \partial_x ((\partial_x u)^2)\, dx -\int
P_{\ge 1} \Lambda^{-\alpha} u \cdot \partial_x u \partial_{xx} u\, dx -
\nu \int (\partial_{xx} u)^2 \,dx \notag \\
& = \frac 12 \int \partial_x P_{<1} \Lambda^{-\alpha} u \cdot
(\partial_x u)^2\, dx - \int P_{\ge 1} \Lambda^{-\alpha }u  \cdot
\partial_x u \partial_{xx} u \, dx- \nu \int (\partial_{xx} u)^2\, dx
\notag \\
& \lesssim \| u(t)\|_{\infty} \| \partial_x u(t) \|_2^2 + \frac 2{\nu} \|
P_{\ge 1} \Lambda^{-\alpha} u (t)\|_{\infty}^2 \| \partial_x u(t)\|_2^2
- \frac{\nu}2 \| \partial_{xx} u(t)\|_2^2 \notag \\
& \lesssim ( \frac 2 {\nu} +1) \big(\|u(t)\|_{\infty}^2 +1\big) \| \partial_x
u(t)\|_2^2 - \frac {\nu}2 \| \partial_{xx} u(t) \|_2^2.
\end{align*}
Since $\|u(t)\|_\infty \le \|u_0 \|_\infty$, a Gronwall in time
argument then yields
\begin{align}
\| \partial_x u(t) \|_2^2 + \nu \int_0^T \| \partial_{xx} u\|_2^2 dt
\lesssim e^{Ct}, \qquad \forall\, t\ge 0. \label{lalala}
\end{align}

Now global wellposedness quickly follows from the continuation
criterion \eqref{Mar6_e9} and \eqref{lalala}, since
\begin{align}
\|\partial_x \Lambda^{-\alpha} u(t) \|_{\infty}& \lesssim \| P_{<1}
\partial_x \Lambda^{-\alpha} u(t) \|_{\infty} + \| P_{\ge 1}
\partial_x \Lambda^{-\alpha} u (t) \|_{\infty} \notag \\
& \lesssim \| u(t) \|_{\infty} + \sum_{N\ge 1} N^{1-\alpha+ \frac
12-2} \| \partial_{xx} u(t) \|_2 \notag \\
& \lesssim \|u_0 \|_\infty + \|\partial_{xx} u(t) \|_2.
                        \label{eq9.50}
\end{align}
\end{rem}

\begin{rem}
The above arguments  can be modified to prove to the global wellposedness of the following 1D model when $\beta=2$:
\begin{equation}
                        \label{eq9.22}
\partial_tu=Hu\partial_xu-\nu\Lambda^{\beta}u,
\end{equation}
where $\nu>0$ is a constant.
This equation has been studied recently in \cite{CCF05}, and later in \cite{D08, LR08}. It is now known that when $\beta\in [1,2)$ the equation is globally wellposed. While in the range $\beta\in [0,1/4)$, evolving from a family of initial data solutions blow up in finite time. In \cite{CCF05}, an additional positivity assumption is imposed on $u_0$. On the other hand, the proof of the global wellposedness in \cite{D08} relies on the non-local maximum principle, which does not work when $\beta=2$ by the same reasoning above.

To deal with this borderline case, we multiply both sides of \eqref{eq9.22} by $u$ and integrate by parts to get
\begin{align*}
\frac 1 2  \frac d {dt}
\int_{\mathbb R} u^2 \,dx
&= \int_{\mathbb R} H u  (\partial_x u) u \,dx+ \nu\int_{\mathbb R} (\partial_{xx}u) u \,dx\\
&= -\frac 1 2\int_{\mathbb R} (\partial_x H u) u^2 \,dx- \nu\int_{\mathbb R} (\partial_{x}u)^2 \,dx\\
&\le \|\partial_x H u(t)\|_2 \|u(t)\|_2\|u(t)\|_\infty- \nu\| (\partial_{x}u(t)\|_2^2\\
&\le \|\partial_xu(t)\|_2 \|u(t)\|_2\|u(0)\|_\infty- \nu\| (\partial_{x}u(t)\|_2^2,
\end{align*}
where in the last inequality, we use the boundedness of the Hilbert transform in $L^2$ and the maximum principle. As before, by Young's inequality and Gronwall's inequality, we get $\|u(t)\|_2\le Ce^{Ct}$, which further implies that $\|Hu(t)\|_2\le Ce^{Ct}$. Therefore, by the classical Sobolev theory, $u$ is globally regular.

Alternatively, we multiply both sides of \eqref{eq9.22} by $-\partial_{xx}u$ and integrate by parts to get
\begin{align*}
&\frac 1 2\frac d {dt} \int (\partial_x u)^2\, dx \\& = - \int (Hu) \frac 12 \partial_x ((\partial_x u)^2)\, dx -
\nu \int (\partial_{xx} u)^2 \,dx \notag \\
& = \frac 12 \int (H\partial_x u)(\partial_x u)^2\, dx - \nu \int (\partial_{xx} u)^2\, dx
\notag \\
& = -\frac 12 \int (H\partial_{xx} u\partial_x u+H\partial_{x} u\partial_{xx}u)u\, dx - \nu \int (\partial_{xx} u)^2\, dx
\notag \\
& \le \| \partial_{xx}u(t)\|_{2} \| \partial_x u (t)\|_2\|u(t)\|
- \nu \| \partial_{xx} u(t)\|_2^2
\end{align*}
Since $\|u(t)\|_\infty \le \|u_0 \|_\infty$, by Young's inequality and Gronwall's inequality, we get
\begin{align*}
\| \partial_x u(t) \|_2^2 + \nu \int_0^T \| \partial_{xx} u\|_2^2 dt
\le Ce^{Ct}, \qquad \forall\, t\ge 0.
\end{align*}
Note that \eqref{eq9.50} still holds when $\alpha=0$, which implies the global regularity of $u$ by the Beale--Kato--Majda criterion.
\end{rem}

\section{Auxiliary lemmas}              \label{sec3}

This section is devoted to several auxiliary lemmas, which will be used in the proofs of Proposition \ref{prop1} and Theorem \ref{thm2} in the following section.

\begin{lem} \label{lem1}
Let $0<\alpha<1$. Assume ${0}<\delta<2$ and
\begin{align*}
g(x)=|x|^{-\delta} \sgn(x) =
\begin{cases}
x^{-\delta}, \quad \text{if $x>0$}, \\
-|x|^{-\delta}, \quad \text{if $x<0$}.
\end{cases}
\end{align*}
Then
\begin{align}
(\Lambda^{-\alpha} g)(x)= C_{\alpha,\delta} |x|^{\alpha-\delta}
\sgn(x), \label{e100}
\end{align}
where $C_{\alpha,\delta}>0$ is a constant depending only on
$(\alpha,\delta)$. Similarly for $0\le \beta<1$, we have
\begin{align}
(\Lambda^\beta g)(x) = C_{\beta,\delta} |x|^{-\delta-\beta} \sgn(x), \label{228_e1}
\end{align}
where $C_{\beta,\delta}>0$ is another constant depending only on $(\beta,\delta)$.
\end{lem}
\begin{proof}[Proof of Lemma \ref{lem1}]
We first note that for any odd function $f=f(y)$,
\begin{align}
(\Lambda^{-\alpha} f)(x) = C_{\alpha} \int_0^\infty f(y) \Bigl(
\frac 1 {|x-y|^{1-\alpha}} - \frac 1 {|x+y|^{1-\alpha}} \Bigr) \,dy, \label{ef_odd}
\end{align}
where $C_{\alpha}>0$ is a constant depending only on $\alpha$.

By a simple scaling argument, we then have for $x>0$,
\begin{align*}
(\Lambda^{-\alpha} g)(x)= C_{\alpha,\delta}^{\prime}
x^{\alpha-\delta} \int_0^{\infty} y^{-\delta} \Bigl( \frac 1
{|1-y|^{1-\alpha}} - \frac 1 {(1+y)^{1-\alpha}} \Bigr) \,dy,
\end{align*}
where $C^{\prime}_{\alpha,\delta}>0$ is a constant. It is easy to
check that the integral in the last equality converges for
$0<\alpha<1$, ${0}<\delta<2$. The identity \eqref{e100} follows
easily.

Next we prove \eqref{228_e1}. Without loss of generality, we can assume $0<\beta<1$. Note that for any odd function
$f$, we have
\begin{align*}
(\Lambda^{\beta} f)(x) =C_{\beta} \int_0^{\infty}
\Bigl( \frac{f(x)-f(y)}{|x-y|^{1+\beta}} + \frac{f(x)+f(y)}
{|x+y|^{1+\beta}} \Bigr) \,dy.
\end{align*}

Again by a scaling argument, we have for $x>0$,
\begin{align*}
(\Lambda^{\beta} g)(x) = C_{\beta,\delta} x^{-\delta-\beta}
\int_0^{\infty} \Bigl( \frac{1-y^{-\delta}}{|1-y|^{1+\beta}} +
\frac{1+y^{-\delta}} {|1+y|^{1+\beta}} \Bigr) \,dy.
\end{align*}
Since $0<\delta<2$ and $0<\beta<1$, it is not difficult to check the last
integral converges. The lemma is proved.
\end{proof}

\begin{lem} \label{lem1a}
Let $0<\alpha<\alpha_1<1$. Define $g$ by
\begin{align*}
g(x)= |x|^{-\alpha_1} \sgn(x) \chi_{|x| \ge 1}
=\begin{cases}
|x|^{-\alpha_1}, \quad x\ge 1,\\
-|x|^{-\alpha_1}, \quad x<-1, \\
0, \quad \text{otherwise}.
\end{cases}
\end{align*}
Then
\begin{align*}
\| \Lambda^{-\alpha} g \|_{L_x^\infty(\mathbb R)} \le C_{\alpha,\alpha_1},
\end{align*}
where $C_{\alpha,\alpha_1}$ is a constant depending only on $(\alpha,\alpha_1)$.
\end{lem}

\begin{proof}[Proof of Lemma \ref{lem1a}]
By using the Littlewood--Paley projectors and noting that $g\in L_x^\infty \cap L_x^{\frac 1 {\alpha_1} +}$, we compute
\begin{align*}
\| \Lambda^{-\alpha} g\|_\infty &\lesssim \| g \|_\infty + \sum_{\text{N dyadic}: N\le 1} \| \Lambda^{-\alpha} P_{N} g \|_{\infty} \\
& \lesssim \|g\|_\infty+ \sum_{N\le 1} N^{-\alpha} N^{\alpha_1-} \| g\|_{\frac 1{\alpha_1}+}\notag\\
& \lesssim \|g\|_\infty+\|g\|_{\frac 1 {\alpha_1}+} \notag\\
& \le C_{\alpha,\alpha_1},
\end{align*}
where in the second inequality we have used the Bernstein inequality. In the third inequality we used
the fact that $\alpha<\alpha_1$ to make the summation over dyadic $N<1$ converge.
\end{proof}

\begin{lem} \label{lem2}
Suppose $0<\alpha<1$. Assume $0<\theta<1-\alpha$. Then for $\lambda
\in \mathbb R$,  $g=g(x)=|x|^{i\lambda -\theta}
=e^{(i\lambda-\theta) \log |x|}$, we have
\begin{align}
(\Lambda^{\alpha} g)(x) = {2^\alpha} \frac{
\Gamma(\frac{1-\theta+i\lambda}2) }{\Gamma ( \frac{\theta-i\lambda}2
) } \cdot \frac{\Gamma(\frac{\theta+\alpha-i\lambda}2)} {\Gamma(
\frac{1-\theta-\alpha+i\lambda}2)} \cdot
|x|^{-(\theta+\alpha-i\lambda)},\quad 0\ne x \in \mathbb R.
\label{e200}
\end{align}
Here 
$\Gamma = \Gamma(z)$ is the usual Gamma function
defined by
\begin{align}
\Gamma(z) = \int_0^{\infty} e^{-s} s^{z-1} ds, \quad \re(z)>0.
\label{e201}
\end{align}

\end{lem}
\begin{proof}[Proof of Lemma \ref{lem2}]
By using \eqref{e201} and a simple change of variable argument, we
have for $0\ne x \in \mathbb R$, $z\in \mathbb C$ with $\re(z)>0$,
\begin{align}
|x|^{-z} = \frac 1 {\Gamma(\frac z 2)} \int_0^{\infty} e^{-s
|x|^2} s^{\frac 12 z -1} ds. \label{e202}
\end{align}
By using \eqref{e202} and the
explicit form of the Fourier transform of Gaussian functions, it is
not difficult to check that
\begin{align}
\int_{-\infty}^{\infty} |x|^{-z} e^{-ix\cdot \xi} \,dx ={\sqrt\pi
2^{1-z}} \cdot \frac{\Gamma(\frac{1-z}2)}{\Gamma(\frac z 2)}
|\xi|^{z-1}, \quad \forall\, 0\ne \xi \in \mathbb R \label{e204}
\end{align}
{provided that $\re (z)\in (0,1)$.}

Take $z=\theta-i\lambda$, then clearly
\begin{align*}
\widehat{|x|^{-(\theta-i\lambda)}}(\xi) = {\sqrt\pi
2^{1-\theta+i\lambda}} \frac{\Gamma ( \frac{1-\theta+i\lambda}2)}
{\Gamma( \frac{\theta-i\lambda}2)} \cdot
|\xi|^{-(1-\theta+i\lambda)}.
\end{align*}
Consequently
\begin{align*}
|\xi|^{\alpha} \widehat{|x|^{-(\theta-i\lambda)}}(\xi) ={\sqrt\pi
2^{1-\theta+i\lambda}} \frac{\Gamma( \frac{1-\theta+i\lambda}2 )}
{\Gamma (\frac{\theta-i\lambda}2)} \cdot
|\xi|^{-(1-\theta-\alpha+i\lambda)}.
\end{align*}

Now note that $0<\theta+\alpha<1$. Using \eqref{e204} again with $z=1-\theta-\alpha+i\lambda$,
we get

\begin{align*}
\mathcal F^{-1} \Bigl( |\xi|^{\alpha}
\widehat{|x|^{i\lambda-\theta}}(\xi) \Bigr) (x) ={2^\alpha} \frac{
\Gamma(\frac{1-\theta+i\lambda}2) }{\Gamma ( \frac{\theta-i\lambda}2
) } \cdot \frac{\Gamma(\frac{\theta+\alpha-i\lambda}2)} {\Gamma(
\frac{1-\theta-\alpha+i\lambda}2)} \cdot
|x|^{-(\theta+\alpha-i\lambda)}.
\end{align*}
This establishes \eqref{e200}.
\end{proof}

The following lemma is crucial for the proof of Proposition \ref{prop1}.

\begin{lem} \label{lem3}
Suppose $0<\alpha<1$ and $0<\theta<1-\alpha$. Then for $\lambda
\in \mathbb R$,  $g=g(x)=|x|^{i\lambda -\theta}$, we have
\begin{align}
-(\partial_x \Lambda^{\alpha} g)(x) = {2^{\alpha+1}}
F_{\alpha,\theta}(\lambda) \; \cdot
x^{-(\theta+\alpha+1-i\lambda)},\quad x>0 \label{e400}
\end{align}
where
\begin{align}
F_{\alpha,\theta}(\lambda)=
\frac{
\Gamma(\frac{1-\theta+i\lambda}2) }{\Gamma ( \frac{\theta-i\lambda}2
) } \cdot \frac{\Gamma(1+\frac{\theta+\alpha-i\lambda}2)} {\Gamma(
\frac{1-\theta-\alpha+i\lambda}2)}, \label{e402}
\end{align}
and also has the sharp bound:
\begin{align}
\frac 1 {C_{\alpha,\theta}} (1+|\lambda|^{\alpha}) \le
 \re\Bigl(F_{\alpha,\theta}(\lambda)\Bigr)
  \le C_{\alpha,\theta} (1+|\lambda|^{\alpha}),
\quad \forall\, \lambda \in \mathbb R. \label{e404}
\end{align}
Here $C_{\alpha,\theta}>0$ {is a constant} depending only on
$(\alpha,\theta)$.
\end{lem}

\begin{proof}[Proof of Lemma \ref{lem3}]
Throughout this proof we will denote by the letter $C$ any constant which depends only
on $(\alpha,\theta)$ but may vary from line to line. Note that \eqref{e400} is a simple
consequence of \eqref{e200} and the fact that $z\Gamma(z)=\Gamma(z+1)$. We first establish the weaker
bound
\begin{align}
\re\Bigl( F_{\alpha,\theta}(\lambda) \Bigr) \ge  \re\Bigl( F_{\alpha,\theta}(0) \Bigr),\qquad \forall\, \lambda \in \mathbb R.
\label{e400_weaker}
\end{align}

To show this we first differentiate g and write
\begin{align*}
-\partial_x g = (\theta-i \lambda) |x|^{i\lambda -\theta-1} \sgn(x).
\end{align*}
Then using the fractional representation of $\Lambda^{\alpha}$ and
a simple scaling argument, we have for any $x>0$,
\begin{align*}
- (\Lambda^\alpha \partial_x g )(x) = C \cdot (\theta-i\lambda) x^{i\lambda -\theta-\alpha-1}
\int_0^{\infty}
\Bigl( \frac{1- y^{i\lambda -\theta-1} } {|1-y|^{1+\alpha}}
+ \frac {1 + y^{i\lambda-\theta-1}} {|1+y|^{1+\alpha}} \Bigr) \,dy.
\end{align*}
Therefore
\begin{align*}
F_{\alpha,\theta}(\lambda) & = C \cdot(\theta-i\lambda) \cdot \int_0^1\Bigl( (1-y)^{-1-\alpha}
(1+y^{\alpha-1} - y^{i\lambda-\theta-1} - y^{-i\lambda+\theta+\alpha}) \notag \\
& \qquad + (1+y)^{-1-\alpha}
(1+y^{\alpha-1} + y^{i\lambda-\theta-1} +y^{-i\lambda+\theta+\alpha})\Bigr)\,dy \notag \\
&=: C \cdot (\theta-i\lambda) G(\lambda).
\end{align*}
By using the binomial expansion and a simple
computation, we have
\begin{align*}
&\;G(\lambda) \notag \\
 & = \sum_{ \substack{k\ge 0,\\ \text{$k$ is even}}} C_{k,\alpha}\cdot
\Bigl( \frac 1 {k+1} + \frac 1 {k+\alpha} - \frac{k+\alpha+1}{k+1}
\bigl( \frac{k+1-\theta} {(k+1-\theta)^2 + \lambda^2} + \frac{k+\theta+\alpha+2}{(k+\theta+\alpha+2)^2+\lambda^2} \bigr)
\Bigr)
\notag \\
& \qquad + i \lambda \sum_{\substack{k\ge 0,\\ \text{$k$ is even}}} C_{k+1,\alpha}
\cdot \bigl( \frac 1 {(k+1-\theta)^2+\lambda^2} - \frac 1 {(k+\theta+\alpha+2)^2 + \lambda^2} \bigr),
\end{align*}
where
$$
C_{k,\alpha}=(-1)^k \binom{-\alpha-1}{k} =\frac{\Gamma(k+\alpha+1)} {\Gamma(k+1) \Gamma(\alpha+1) } =
\frac{(\alpha+1) \cdots (\alpha+k)}{k!}
$$
Using the asymptotics
\begin{align*}
\Bigl| \binom{-\alpha-1} {k} \Bigr| \sim const \cdot k^{\alpha}, \quad k\gtrsim 1,
\end{align*}
it is not difficult to check the series representation of $G(\lambda)$ converges.

Clearly then
\begin{align*}
&\;\re\Bigl( F_{\alpha,\theta}(\lambda) \Bigr) \notag \\
& = \theta \sum_{ \substack{k\ge 0,\\ \text{$k$ is even}}} C_{k,\alpha}
\Bigl( \frac 1 {k+1} + \frac 1 {k+\alpha} - \frac{k+\alpha+1}{k+1}
\bigl( \frac{k+1-\theta} {(k+1-\theta)^2 + \lambda^2} + \frac{k+\theta+\alpha+2}{(k+\theta+\alpha+2)^2+\lambda^2} \bigr)
\Bigr)
\notag \\
& \qquad + \lambda^2 \sum_{ \substack{k\ge 0, \\\text{$k$ is even}}} C_{k+1,\alpha}
\cdot \bigl( \frac 1 {(k+1-\theta)^2+\lambda^2} - \frac 1 {(k+\theta+\alpha+2)^2 + \lambda^2} \bigr),
\end{align*}
Owing to our assumption $0<\theta<1-\alpha$, $0<\alpha<1$, it is easy to check directly from the above
expression that \eqref{e400_weaker} holds. Now by \eqref{e402} and \eqref{e400_weaker}, we obtain
\begin{align*}
\re( F(\lambda)) \ge C>0, \quad \forall\, \lambda \in \mathbb R.
\end{align*}

It remains to prove the bound \eqref{e404} for $\lambda$ sufficiently large. For this we need to use the
Stirling's formula for the Gamma function. It states that for any complex $z$ with $|arg(z)| < \pi -\epsilon$
(here $arg(z)$ takes values in $[-\pi,\pi)$) and $|z|\gg 1$, we have
\begin{align*}
\ln \Gamma (z) = (z-\frac 1 2 ) \ln z -z + \frac 12 \ln(2\pi) + \frac 1 {12 z} + O(|z|^{-2}).
\end{align*}

Using \eqref{e402} and a tedious computation, we arrive at
\begin{align*}
\ln F(\lambda) = O(\lambda^{-2}) + (1+\alpha) \ln ( \frac 12 \lambda) - i \frac {\pi}2
+\frac {i}{4\lambda} (\theta^2 + 4\theta(1+\alpha) + 2\alpha(1+\alpha)),
\end{align*}
where $O(\lambda^{-2})$ denotes the remainder term (complex-valued)
whose absolute value is less than $\lambda^{-2}$. It follows easily
that

\begin{align*}
F(\lambda) = ( \frac 1 2 \lambda)^{1+\alpha} \Bigl( -i \cos ( \frac {C_1}{\lambda}) + \sin(\frac{C_1} {\lambda})
\Bigr) + O(\lambda^{-1+\alpha}),
\end{align*}
where $C_1>0$ is a constant depending only on $(\alpha,\theta)$. Clearly
\begin{align*}
\re(F(\lambda)) = C_2 \lambda^{\alpha} + O(\lambda^{-1+\alpha}),\quad \lambda\gg 1.
\end{align*}
where $C_2>0$ is another constant. The sharp bound \eqref{e404} follows.
\end{proof}

\section{Finite time singularities}                             \label{sec4}

In this section, we complete the proofs of Proposition \ref{prop1}
and Theorem \ref{thm2}.  The proof of Proposition \ref{prop1} is
inspired by an argument in \cite{CCF05} by using Mellin transforms
and the corresponding Parseval identity. See also \cite{CCF06,
dongli2, LR08, LR09}.
\begin{proof}[Proof of Proposition \ref{prop1}]
Since by assumption $u$ is odd, using Lemma \ref{lem1}, we get
\begin{align*}
\int_0^{\infty} \frac{ \Lambda^{-\alpha} ( \Lambda^{-\alpha} u \partial_x u ) } {x^{\delta}} \,dx
& = \frac 12 \int_{\infty}^\infty \Lambda^{-\alpha} ( \Lambda^{-\alpha} u \partial_x u )
|x|^{-\delta} \sgn (x) \,dx \\
& = \frac 12 \int_{-\infty}^{\infty} \Lambda^{-\alpha} u \partial_x u \Lambda^{-\alpha}
(|x|^{-\delta} \sgn(x)) \,dx \notag \\
& = C_{\alpha,\delta} \int_0^\infty \Lambda^{-\alpha} u \cdot \partial_x u\cdot x^{\alpha-\delta} \,dx.
\end{align*}

By using the Parseval identity for Mellin transforms, we have
\begin{align}
\int_0^{\infty} \Lambda^{-\alpha} u \partial_x u x^{\alpha-\delta} \,dx
& = \int_0^\infty \frac{\Lambda^{-\alpha} u }{x^{\frac {\delta}2}} \cdot
\frac{\partial_x u}{x^{\frac {\delta}2 - \alpha-1}} \frac {dx} x \notag \\
& = \frac 1 {2\pi} \int_{-\infty}^{\infty} \overline{A(\lambda)}
B(\lambda) \,d\lambda, \label{227_e1}
\end{align}
where
\begin{align}
A(\lambda)& = \int_0^{\infty} \Lambda^{-\alpha} u \cdot x^{i\lambda - \frac {\delta}2 -1} \,dx, \label{227_e2} \\
B(\lambda) & = \int_0^{\infty} \partial_x u \cdot x^{i\lambda - \frac {\delta}2 +\alpha} \,dx \notag \\
& = \int_0^\infty \partial_x \Lambda^{\alpha} ( \Lambda^{-\alpha} u ) \cdot x ^{i\lambda - \frac {\delta}2 +\alpha}
\,dx. \notag
\end{align}

By Lemma \ref{lem3} and observing that $\partial_x u $ is an even function on $\mathbb R$, we obtain
(note here $\theta=\frac{\delta}2 -\alpha$ and $0<\theta<1-\alpha$),
\begin{align}
B(\lambda) & = {2^{\alpha+1}} F_{\alpha,\theta}(\lambda)
\int_0^\infty
(\Lambda^{-\alpha} u)\;\cdot x^{-(\frac {\delta}2 +1 -i \lambda)} \,dx \notag \\
& = {2^{\alpha+1}} F_{\alpha,\theta}(\lambda) A(\lambda),
\label{227_e3}
\end{align}
where 
$F_{\alpha,\theta}(\lambda)$
was defined in \eqref{e402}. Substituting \eqref{227_e3} into \eqref{227_e1} and using \eqref{e404}, we have
\begin{align*}
\int_0^{\infty} \Lambda^{-\alpha} u \partial_x u x^{\alpha-\delta} \,dx
& = C_{\alpha,\theta} \int_{-\infty}^{\infty}
\re( F_{\alpha,\theta}(\lambda)) |A(\lambda)|^2 \,d\lambda \\
& \ge C_{\alpha,\theta} \int_{-\infty}^{\infty} |A(\lambda)|^2 \,d\lambda \\
& \ge C_{\alpha,\theta} \int_0^{\infty} \frac { (\Lambda^{-\alpha}u)^2} {x^{1+\delta}} \,dx,
\end{align*}
where the last step follows from \eqref{227_e2} and the Parseval identity for Mellin transforms. This completes the proof of \eqref{e25_40a}.
\end{proof}

\begin{lem}[$L_x^1$ norm is nonincreasing] \label{lem100}
Assume $0<\alpha<1$ and $0\le \beta \le 2$ in \eqref{e_main}.
Let the initial data $u_0 \in L_x^1(\mathbb R) \cap H_x^1(\mathbb R)$ be such
that $u_0(x)\ge 0$ for any $x\ge 0$ and $u_0$ is odd in $x$. Let $u=u(t,x)$ be the corresponding solution
to \eqref{e_main} with lifespan $[0,T_0)$ where $0<T_0\le +\infty$. Then for any $t\in [0,T_0)$,
we have $u(t,x)\ge 0$ for any $x\ge 0$, and
\begin{align*}
\|u(t,\cdot)\|_{L_x^1(\mathbb R)} \le \|u_0\|_{L_x^1(\mathbb R)}.
\end{align*}
\end{lem}
\begin{proof}[Proof of Lemma \ref{lem100}]
We shall use the idea of time-splitting approximation. Namely the solution
$u=u(t,x)$ on any $[0,T^\prime]$ with $T^\prime<T_0$ can be approximated by interlacing the nonlinear
evolution $\partial_t f = \Lambda^{-\alpha} f \partial_x f$ with the linear evolution
$\partial_t g= -\nu \Lambda^{\beta} g$ with small time step $h=T^\prime/N$ where $N\to \infty$. Consider
first the linear evolution $\partial_t g=-\nu \Lambda^{\beta} g$ on a time interval $[0,h]$ with
$g_0$ odd in $x$ and $g_0(x)\ge 0$ for any $x\ge 0$. Then clearly $g=g(t,x)$ is also odd in $x$ and
has the representation
\begin{align*}
g(t,x) = \int_0^{\infty} \Bigl( k(t,x-y) - k(t,x+y) \Bigr) g_0(y) \,dy,
\end{align*}
where $k(t,z)$ is the fundamental solution corresponding to the
propagator $\partial_t +\nu \Lambda^{\beta}$, {which is nonnegative
and radially decreasing}. It is easy to check that $g(t,x)\ge 0$ for
any $x\ge 0$. Furthermore
\begin{align*}
\int_0^\infty g(t,x) \,dx & = \int_0^\infty \Bigl( \int_{-y}^y k(t,z) dz \Bigr) g_0(y) \,dy \notag\\
& \le \int_0^\infty g_0(y)\,dy,
\end{align*}
where we used the fact that $k(t,z)$ is nonnegative and $\int_{-\infty}^{\infty} k(t,z) dz =1$.
We now check the nonlinear evolution $\partial_t f = \Lambda^{-\alpha} f \partial_x f$ on the
time interval $[0,h]$. Assume $f_0$ is odd in $x$ and $f_0(x)\ge 0$ for any $x\ge 0$. It is not difficult
to check that $f(t,x)$ is odd in $x$ and $f(t,x) \ge 0$ for any $x\ge 0$. Then by using \eqref{ef_odd},
integrating by parts, we compute (here for simplicity of presentation
we omit any $\epsilon$-regularization argument needed for the
convergence of integrals):
\begin{align*}
&\;\;\frac d {dt} \int_0^{\infty} f(t,x) \,dx \notag \\
& = \int_0^{\infty} \Lambda^{-\alpha} f \partial_x f \,dx \notag \\
& = \int_0^{\infty} \int_0^{\infty} (\partial_x f)(t,x) f(t,y) \cdot
\Bigl( |x-y|^{-(1-\alpha)} - |x+y|^{-(1+\alpha)}\Bigr) \,dx \,dy \notag \\
& = - \int_0^{\infty} \int_0^{\infty} f(t,x) f(t,y)
\Bigl( - (1-\alpha) |x-y|^{-(2-\alpha)} \sgn(x-y) \notag \\
& \qquad\qquad+ (1+\alpha)\cdot |x+y|^{-(2+\alpha)} \Bigr) \,dx \,dy \notag \\
& = - \int_0^{\infty} \int_0^{\infty} f(t,x) f(t,y) \cdot (1+\alpha) \cdot |x+y|^{-(2+\alpha)} \,dx \,dy,
\end{align*}
where in the last equality we have used a symmetrization in $x$ and $y$ to make the first integral corresponding
to the kernel $|x-y|^{-(2-\alpha)} \sgn(x-y)$ vanish. Hence we obtain $\|f(t,\cdot)\|_{L_x^1(\mathbb R^+)}
\le \| f_0 \|_{L_x^1(\mathbb R^+)}$.  This completes the proof for the nonlinear evolution part.
\end{proof}

We are now ready to complete the
\begin{proof}[Proof of Theorem \ref{thm2}]
We will argue by contradiction. Assume the solution corresponding to $u_0$ exists for all time, then by
Lemma \ref{lem1}, we have
\begin{align*}
\int_0^{\infty} \frac{(\Lambda^{-\alpha} u)(t,x)} {x^{\delta}} \,dx &= C_{\alpha,\delta} \int_0^{\infty}
\frac{u(t,x)} {x^{\delta-\alpha}} \,dx \notag \\
& \lesssim \|u(t,\cdot)\|_{L_x^1} + \int_0^1 \frac{ \|\partial_x u\|_{\infty}} {x^{\delta-\alpha-1}} \,dx \notag\\
& \lesssim \| u(t,\cdot) \|_{L_x^1} + \| \partial_x u \|_{L_x^\infty} <+\infty.
\end{align*}

This shows that the integral $\int_0^\infty (\Lambda^{-\alpha}u )(t,x)  x^{-\delta} \,dx$ is finite for any $t\ge0$.
Next we show that the integral blows up at some finite $T>0$ and obtain contradiction.

By using \eqref{e_main}, Proposition \ref{prop1} and Lemma \ref{lem1}, we have
\begin{align}
\frac d {dt} \int_0^{\infty} \frac{(\Lambda^{-\alpha} u)(t,x)} {x^{\delta}} \,dx &=
\int_0^{\infty} \frac{\Lambda^{-\alpha} ( \Lambda^{-\alpha} u \partial_x u ) } {x^{\delta}} \,dx
- \nu \int_0^{\infty} \frac{\Lambda^{\beta} \Lambda^{-\alpha} u } {x^{\delta}} \,dx \notag \\
& \ge C_{\alpha,\delta} \int_0^{\infty} \frac{(\Lambda^{-\alpha} u )^2} {x^{1+\delta}} \,dx
- \nu C_{\beta,\delta} \int_0^{\infty} \frac{\Lambda^{-\alpha} u }{x^{\delta+\beta}} \,dx. \label{e228_e3}
\end{align}

By Cauchy--Schwartz, we obtain
\begin{align*}
&\int_0^{\infty} \frac{\Lambda^{-\alpha} u }{x^{\delta+\beta}} \,dx \notag \\
 =&
\int_0^1 \frac{\Lambda^{-\alpha} u }{x^{\delta+\beta}} \,dx + \int_1^\infty
\frac{\Lambda^{-\alpha} u }{ x^{\delta+\beta}} \,dx \notag \\
\le & \Bigl( \int_0^1 \frac{(\Lambda^{-\alpha} u)^2} {x^{1+\delta}} \,dx \Bigr)^{\frac 12}
\Bigl( \int_0^1 \frac 1 {x^{\delta+2\beta-1}} \,dx \Bigr)^{\frac 12}
\notag \\
& \qquad + \| u(t) \|_{L_x^1} \cdot \| \Lambda^{-\alpha}
( |x|^{-\delta-\beta} \sgn(x) \chi_{|x|\ge 1} ) \|_{L_x^\infty}.
\end{align*}

Observe that $\delta+2\beta-1<1$ and $\delta+\beta>\alpha$. By Lemma \ref{lem1a} and Lemma \ref{lem100}, we get
\begin{align*}
\int_0^{\infty} \frac{\Lambda^{-\alpha} u } {x^{\delta+\beta}} \,dx
\le \frac {C_{\alpha,\delta}} {2(1+\nu)} \int_0^1 \frac
{(\Lambda^{-\alpha} u )^2} {x^{1+\delta}} \,dx +
C_{\alpha,\delta,\beta} ({1+}\| u_0 \|_{L_x^1}).
\end{align*}

Substituting this last estimate into \eqref{e228_e3} and using again Lemma \ref{lem1a} and Cauchy--Schwartz, we obtain
\begin{align*}
\frac d {dt} \int_0^{\infty} \frac{ (\Lambda^{-\alpha} u )(t,x)} {x^{\delta}} \,dx
& \ge \frac {C_{\alpha,\delta}}2 \int_0^1 \frac{(\Lambda^{-\alpha} u )^2} {x^{1+\delta}} \,dx
-C_{\alpha,\delta,\beta,\nu} ( 1+ \|u_0\|_{L_x^1} ) \notag \\
& \ge  C_{\alpha,\delta}^{\prime} \Bigl( \int_0^1 \frac{(\Lambda^{-\alpha} u)(t,x)} {x^\delta} \,dx \Bigr)^2
-C_{\alpha,\delta,\beta,\nu} ( 1+ \|u_0\|_{L_x^1} ) \notag \\
& \ge C_{\alpha,\delta}^{\prime\prime}  \Bigl( \int_0^{\infty}
 \frac{(\Lambda^{-\alpha} u)(t,x)} {x^\delta} \,dx \Bigr)^2
 - C_{\alpha,\delta,\beta,\nu}^{\prime} (1+ \|u_0\|_{L_x^1})^2.
\end{align*}
It is now clear that if $u_0$ satisfy \eqref{e228_e5} with sufficiently large constant $C_{\alpha,\beta,\delta,\nu}$,
then we obtain the inequality of the form
\begin{align*}
\frac d {dt} a(t) \ge C \cdot a(t)^2,
\end{align*}
where
$$
a(t) = \int_0^\infty \frac{(\Lambda^{-\alpha} u )(t,x)} {x^{\delta}} \,dx.
$$
Clearly $a(t)$ goes to infinity
in finite time. We have obtained the desired contradiction and the proof is now completed.

\end{proof}

\end{document}